\documentclass[12pt]{amsart}       
\usepackage{txfonts}
\usepackage{amssymb}
\usepackage{eucal}
\usepackage{graphicx}
\usepackage{amsmath}
\usepackage{amscd}
\usepackage[all]{xy}           

\usepackage{amsfonts,latexsym}
\usepackage{xspace}
\usepackage{epsfig}
\usepackage{float}
\usepackage{color}
\usepackage{fancybox}
\usepackage{colordvi}
\usepackage{multicol}
\usepackage{colordvi}
\usepackage{pdfsync}

\usepackage[active]{srcltx} 
\usepackage[colorlinks,final,backref=page,hyperindex,hypertex]{hyperref}

\newcommand{\nc}{\newcommand}
\newcommand{\delete}[1]{}

\nc{\mlabel}[1]{\label{#1}}  
\nc{\mcite}[1]{\cite{#1}}  
\nc{\mref}[1]{\ref{#1}}  
\nc{\mbibitem}[1]{\bibitem{#1}} 

\delete{
\nc{\mlabel}[1]{\label{#1}  
{\hfill \hspace{1cm}{\small\tt{{\ }\hfill(#1)}}}}
\nc{\mcite}[1]{\cite{#1}{\small{\tt{{\ }(#1)}}}}  
\nc{\mref}[1]{\ref{#1}{{\tt{{\ }(#1)}}}}  
\nc{\mbibitem}[1]{\bibitem[\bf #1]{#1}} 
}

\newtheorem{theorem}{Theorem}[section]
\newtheorem{defn}[theorem]{Definition}
\newtheorem{lemma}[theorem]{Lemma}
\newtheorem{coro}[theorem]{Corollary}
\newtheorem{prop-def}{Proposition-Definition}[section]

\newtheorem{tempex}[theorem]{Example}
\newtheorem{tempexs}[theorem]{Examples}
\newtheorem{temprmk}[theorem]{Remark}
\newtheorem{tempexer}{Exercise}[section]

\nc{\vsa}{\vspace{-.1cm}} \nc{\vsb}{\vspace{-.2cm}}
\nc{\vsc}{\vspace{-.3cm}} \nc{\vsd}{\vspace{-.4cm}}
\nc{\vse}{\vspace{-.5cm}}

\nc{\UN}{U_{N}}
\nc{\FN}{F_{\mathrm N}}
\nc{\dfgen}{V} \nc{\dfrel}{R}
\nc{\dfgenb}{\vec{v}} \nc{\dfrelb}{\vec{r}}
\nc{\dfgene}{v} \nc{\dfrele}{r}
\nc{\dfop}{\odot}
\nc{\dfoa}{\dfop^{(1)}} \nc{\dfob}{\dfop^{(2)}}
\nc{\dfoc}{\dfop^{(3)}} \nc{\dfod}{\dfop^{(4)}}

\nc{\disp}[1]{\displaystyle{#1}}
\nc{\bin}[2]{ (_{\stackrel{\scs{#1}}{\scs{#2}}})}  
\nc{\binc}[2]{ \left (\!\! \begin{array}{c} \scs{#1}\\
    \scs{#2} \end{array}\!\! \right )}  
\nc{\bincc}[2]{  \left ( {\scs{#1} \atop
    \vspace{-.5cm}\scs{#2}} \right )}  
\nc{\sarray}[2]{\begin{array}{c}#1 \vspace{.1cm}\\ \hline
    \vspace{-.35cm} \\ #2 \end{array}}
\nc{\bs}{\bar{S}} \nc{\ep}{\epsilon}
\nc{\dbigcup}{\stackrel{\bullet}{\bigcup}}
\nc{\la}{\longrightarrow} \nc{\cprod}{\ast} \nc{\rar}{\rightarrow}
\nc{\dar}{\downarrow} \nc{\labeq}[1]{\stackrel{#1}{=}}
\nc{\dap}[1]{\downarrow \rlap{$\scriptstyle{#1}$}}
\nc{\uap}[1]{\uparrow \rlap{$\scriptstyle{#1}$}}
\nc{\defeq}{\stackrel{\rm def}{=}} \nc{\dis}[1]{\displaystyle{#1}}
\nc{\dotcup}{\ \displaystyle{\bigcup^\bullet}\ }
\nc{\sdotcup}{\tiny{ \displaystyle{\bigcup^\bullet}\ }}
\nc{\fe}{\'{e}}
\nc{\hcm}{\ \hat{,}\ } \nc{\hcirc}{\hat{\circ}}
\nc{\hts}{\hat{\shpr}} \nc{\lts}{\stackrel{\leftarrow}{\shpr}}
\nc{\denshpr}{\den{\shpr}}
\nc{\rts}{\stackrel{\rightarrow}{\shpr}} \nc{\lleft}{[}
\nc{\lright}{]} \nc{\uni}[1]{\tilde{#1}} \nc{\free}[1]{\bar{#1}}
\nc{\freea}[1]{\tilde{#1}} \nc{\freev}[1]{\hat{#1}}
\nc{\dt}[1]{\hat{#1}}
\nc{\wor}[1]{\check{#1}}
\nc{\intg}[1]{F_C(#1)}
\nc{\den}[1]{\check{#1}} \nc{\lrpa}{\wr} \nc{\mprod}{\pm}
\nc{\dprod}{\ast_P} \nc{\curlyl}{\left \{ \begin{array}{c} {} \\
{} \end{array}
    \right .  \!\!\!\!\!\!\!}
\nc{\curlyr}{ \!\!\!\!\!\!\!
    \left . \begin{array}{c} {} \\ {} \end{array}
    \right \} }
\nc{\longmid}{\left | \begin{array}{c} {} \\ {} \end{array}
    \right . \!\!\!\!\!\!\!}
\nc{\lin}{\call} \nc{\ot}{\otimes}
\nc{\ora}[1]{\stackrel{#1}{\rar}}
\nc{\ola}[1]{\stackrel{#1}{\la}}
\nc{\scs}[1]{\scriptstyle{#1}} \nc{\mrm}[1]{{\rm #1}}
\nc{\margin}[1]{\marginpar{\rm #1}}   
\nc{\dirlim}{\displaystyle{\lim_{\longrightarrow}}\,}
\nc{\invlim}{\displaystyle{\lim_{\longleftarrow}}\,}
\nc{\mvp}{\vspace{0.5cm}}
\nc{\mult}{m}       
\nc{\svp}{\vspace{2cm}} \nc{\vp}{\vspace{8cm}}
\nc{\proofbegin}{\noindent{\bf Proof: }}
\nc{\proofend}{$\blacksquare$ \vspace{0.5cm}}
\nc{\sha}{{\mbox{\cyr X}}}  
\nc{\ncsha}{{\mbox{\cyr X}^{\mathrm NC}}} \nc{\ncshao}{{\mbox{\cyr
X}^{\mathrm NC,\,0}}}
\nc{\shpr}{\diamond}    
\nc{\shprc}{\shpr_c}
\nc{\shpro}{\diamond^0}    
\nc{\shpru}{\check{\diamond}} \nc{\spr}{\cdot}
\nc{\catpr}{\diamond_l} \nc{\rcatpr}{\diamond_r}
\nc{\lapr}{\diamond_a} \nc{\lepr}{\diamond_e} \nc{\sprod}{\bullet}
\nc{\un}{u}                 
\nc{\vep}{\varepsilon} \nc{\labs}{\mid\!} \nc{\rabs}{\!\mid}
\nc{\hsha}{\widehat{\sha}} \nc{\lsha}{\stackrel{\leftarrow}{\sha}}
\nc{\rsha}{\stackrel{\rightarrow}{\sha}} \nc{\lc}{\lfloor}
\nc{\rc}{\rfloor}
\nc{\sqmon}[1]{\langle #1\rangle}
\nc{\altx}{\Lambda}
\nc{\vecT}{\vec{T}}
\nc{\piword}{{\mathfrak P}}

\nc{\mmbox}[1]{\mbox{\ #1\ }}
\nc{\ayb}{\mrm{AYB}} \nc{\mayb}{\mrm{mAYB}} \nc{\cyb}{\mrm{cyb}}
\nc{\ann}{\mrm{ann}} \nc{\Aut}{\mrm{Aut}} \nc{\cabqr}{\mrm{CABQR
}} \nc{\can}{\mrm{can}} \nc{\colim}{\mrm{colim}}
\nc{\Cont}{\mrm{Cont}} \nc{\rchar}{\mrm{char}}
\nc{\cok}{\mrm{coker}} \nc{\dtf}{{R-{\rm tf}}} \nc{\dtor}{{R-{\rm
tor}}}

\nc{\Div}{{\mrm Div}} \nc{\End}{\mrm{End}} \nc{\Ext}{\mrm{Ext}}
\nc{\FG}{\mrm{FG}} \nc{\Fil}{\mrm{Fil}} \nc{\Frob}{\mrm{Frob}}
\nc{\Gal}{\mrm{Gal}} \nc{\GL}{\mrm{GL}} \nc{\Hom}{\mrm{Hom}}
\nc{\hsr}{\mrm{H}} \nc{\hpol}{\mrm{HP}} \nc{\id}{\mrm{id}}
\nc{\im}{\mrm{im}} \nc{\incl}{\mrm{incl}} \nc{\Loday}{\mrm{ABQR}\
} \nc{\length}{\mrm{length}} \nc{\LR}{\mrm{LR}} \nc{\mchar}{\rm
char} \nc{\pmchar}{\partial\mchar} \nc{\map}{\mrm{Map}}
\nc{\MS}{\mrm{MS}} \nc{\OS}{\mrm{OS}} \nc{\NC}{\mrm{NC}}
\nc{\rba}{\rm{Rota-Baxter algebra}\xspace}
\nc{\rbas}{\rm{Rota-Baxter algebras}\xspace}
\nc{\rbw}{\rm{RBW}\xspace}
\nc{\rbws}{\rm{RBWs}\xspace}
\nc{\rbadj}{\rm{RB}\xspace}
\nc{\mpart}{\mrm{part}} \nc{\ql}{{\QQ_\ell}} \nc{\qp}{{\QQ_p}}
\nc{\rank}{\mrm{rank}} \nc{\rcot}{\mrm{cot}} \nc{\rdef}{\mrm{def}}
\nc{\rdiv}{{\rm div}} \nc{\rtf}{{\rm tf}} \nc{\rtor}{{\rm tor}}
\nc{\res}{\mrm{res}} \nc{\SL}{\mrm{SL}} \nc{\Spec}{\mrm{Spec}}
\nc{\tor}{\mrm{tor}} \nc{\Tr}{\mrm{Tr}}
\nc{\mtr}{\mrm{tr}}

\nc{\ab}{\mathbf{Ab}} \nc{\Alg}{\mathbf{Alg}}
\nc{\Bax}{\mathbf{CRB}} \nc{\Algo}{\mathbf{Alg}^0}
\nc{\cRB}{\mathbf{CRB}} \nc{\cRBo}{\mathbf{CRB}^0}
\nc{\RBo}{\mathbf{RB}^0} \nc{\BRB}{\mathbf{RB}}
\nc{\Dend}{\mathbf{DD}} \nc{\bfk}{{\bf k}} \nc{\bfone}{{\bf 1}}
\nc{\base}[1]{{a_{#1}}} \nc{\Cat}{\mathbf{Cat}}
 \nc{\NS}{\mathbf{NS}}
\nc{\NA}{\mathbf{NA}}
\nc{\ND}{\mathbf{ND}}
\nc{\Diff}{\mathbf{Diff}} \nc{\gap}{\marginpar{\bf
Incomplete}\noindent{\bf Incomplete!!}
    \svp}
\nc{\FMod}{\mathbf{FMod}} \nc{\Int}{\mathbf{Int}}
\nc{\Mon}{\mathbf{Mon}}
\nc{\RB}{\mathbf{RB}} \nc{\remarks}{\noindent{\bf Remarks: }}
\nc{\Rep}{\mathbf{Rep}} \nc{\Rings}{\mathbf{Rings}}
\nc{\Sets}{\mathbf{Sets}} \nc{\bfx}{\mathbf{x}}
\nc{\BA}{{\Bbb A}} \nc{\CC}{{\Bbb C}} \nc{\DD}{{\Bbb D}}
\nc{\EE}{{\Bbb E}} \nc{\FF}{{\Bbb F}} \nc{\GG}{{\Bbb G}}
\nc{\HH}{{\Bbb H}} \nc{\LL}{{\Bbb L}} \nc{\NN}{{\Bbb N}}
\nc{\QQ}{{\Bbb Q}} \nc{\RR}{{\Bbb R}} \nc{\TT}{{\Bbb T}}
\nc{\VV}{{\Bbb V}} \nc{\ZZ}{{\Bbb Z}}

\nc{\cala}{{\mathcal A}} \nc{\calc}{{\mathcal C}}
\nc{\cald}{{\mathcal D}} \nc{\cale}{{\mathcal E}}
\nc{\calf}{{\mathcal F}} \nc{\calg}{{\mathcal G}}
\nc{\calh}{{\mathcal H}} \nc{\cali}{{\mathcal I}}
\nc{\calj}{{\mathcal J}} \nc{\call}{{\mathcal L}}
\nc{\calm}{{\mathcal M}} \nc{\caln}{{\mathcal N}}
\nc{\calo}{{\mathcal O}} \nc{\calp}{{\mathcal P}}
\nc{\calr}{{\mathcal R}} \nc{\calt}{{\mathcal T}}
\nc{\calw}{{\mathcal W}} \nc{\calx}{{\mathcal X}}
\nc{\CA}{\mathcal{A}}

\nc{\frakA}{{\mathfrak A}}
\nc{\fraka}{{\mathfrak a}}
\nc{\frakB}{{\mathfrak B}}
\nc{\frakb}{{\mathfrak b}}
\nc{\frakd}{{\mathfrak d}}
\nc{\frakF}{{\mathfrak F}}
\nc{\frakg}{{\mathfrak g}}
\nc{\frakm}{{\mathfrak m}}
\nc{\frakM}{{\mathfrak M}}
\nc{\frakMo}{{\mathfrak M}^0}
\nc{\frakP}{{\mathfrak P}}
\nc{\frakp}{{\mathfrak p}}
\nc{\frakS}{{\mathfrak S}}
\nc{\frakSo}{{\mathfrak S}^0}
\nc{\fraks}{{\mathfrak s}}
\nc{\os}{\overline{\fraks}}
\nc{\frakT}{{\mathfrak T}}
\nc{\frakTo}{{\mathfrak T}^0}
\nc{\oT}{\overline{T}}
\nc{\frakX}{{\mathfrak X}}
\nc{\frakXo}{{\mathfrak X}^0}
\nc{\frakx}{{\mathbf x}}
\nc{\frakTx}{\frakT}      
\nc{\frakTa}{\frakT^a}        
\nc{\frakTxo}{\frakTx^0}   
\nc{\caltao}{\calt^{a,0}}   
\nc{\ox}{\overline{\frakx}}
\nc{\fraky}{{\mathfrak y}}
\nc{\frakz}{{\mathfrak z}}
\nc{\oX}{\overline{X}}

\font\cyr=wncyr10 

\begin{document}

\nc{\tred}[1]{\textcolor{red}{#1}}
\nc{\li}[1]{\tred{Li:#1 }}

\title{Nijenhuis algebras, NS algebras and N-dendriform
algebras}

\author{Peng Lei}
\address{
Department of Mathematics, Lanzhou University,
Lanzhou, Gansu 730000, China}
\email{leip@lzu.edu.cn}

\author{Li Guo}
\address{
Department of Mathematics and Computer Science,
Rutgers University,
Newark, NJ 07102, USA}
\email{liguo@newark.rutgers.edu}



\begin{abstract}
In this paper we study (associative) Nijenhuis algebras, with emphasis
on the relationship between the category of Nijenhuis algebras and
the categories of NS algebras. This is in analogy
to the well-known theory of the adjoint functor from the category of
Lie algebras to that of associative algebras, and the more recent
results on the adjoint functor from the categories of dendriform and tridendriform algebras to that of Rota-Baxter algebras. We first give an
explicit construction of free Nijenhuis algebras and
then apply it to obtain the universal enveloping Nijenhuis algebra of an NS algebra. We further apply the construction to determine the binary quadratic nonsymmetric algebra, called the N-dendriform algebra, that is compatible with the Nijenhuis algebra. As it turns out, the N-dendriform algebra has more relations than the NS algebra.
\end{abstract}

\maketitle

\tableofcontents

\setcounter{section}{0}

\section{Introduction}
\mlabel{sec:intro}

Through the antisymmetry bracket $[x,y]:=xy-yx$, an associative algebra $A$ defines a Lie algebra structure on $A$.
The resulting functor from the category of associative algebras to that of Lie algebras and its adjoint functor have played a fundamental role in the study of these algebraic structures.
A similar relationship holds for Rota-Baxter algebras and dendriform algebras.


This paper studies a similar relationship between (associative) Nijenhuis algebras and NS algebras.
\medskip

A {\bf Nijenhuis algebra} is a nonunitary associative algebra $N$ with a linear endomorphism $P$ satisfying the {\bf Nijenhuis equation}:

\begin{equation}
    P(x)P(y) = P(P(x)y) + P(xP(y)) - P^2(xy),\quad \forall x,y \in N.
    \mlabel{eq:Nij}
\end{equation}
The concept of a Nijenhuis operator on a Lie algebra originated from the important concept of a Nijenhuis tensor that was introduced by Nijenhuis~\mcite{N} in the study of pseudo-complex manifolds in the 1950s and was related to the well-known concepts of Schouten-Nijenhuis bracket, the Fr\"olicher-Nijenhuis bracket~\mcite{FN} and the Nijenhuis-Richardson bracket. Nijenhuis operator operators on a Lie algebra appeared in~\mcite{KM} in a more general study of Poisson-Nijenhuis manifolds and then more recently in~\mcite{GS1,GS2} in the context of the classical Yang-Baxter equation.

The Nijenhuis operator on an associative algebra was introduced by Carinena and coauthors~\mcite{CGM} to study quantum bi-Hamiltonian systems.
In~\mcite{Uc}, Nijenhuis operators are constructed by analogy with Poisson-Nijenhuis geometry, from relative Rota-Baxter operators.

Note the close analogue of the Nijenhuis operator with the more familiar {\bf Rota-Baxter operator} of weight $\lambda$ (where $\lambda$ is a constant) defined to be a linear endomorphism $P$ on an associative algebra $R$ satisfying
$$ P(x)P(y)=P(P(x)y) + P(xP(y)) +\lambda P(xy), \quad \forall x, y\in R.$$
The latter originated from the probability study of G. Baxter~\mcite{Ba}, was studied by Cartier and Rota and is closely related to the operator form of the classical Yang-Baxter equation. Its study has experienced a quite remarkable renascence in the last decade with many applications in mathematics and physics, most notably the work of Connes and Kreimer on renormalization of quantum field theory~\mcite{CK,EGK,EGM}. See~\mcite{Gub} for further details and references.

The recent theoretic developments of Nijenhuis algebras have largely followed those of Rota-Baxter algebras. Commutative Nijenhuis algebras were constructed in~\mcite{EF2,EL} following the construction of free commutative Rota-Baxter algebras~\mcite{GK1}.

Another development followed the relationship between Rota-Baxter algebras and dendriform algebras.
Recall that a {\bf dendriform algebra}, defined by Loday~\mcite{Lo1}, is a vector space $D$ with two binary operations $\prec$ and $\succ$ such that
$$(x\prec y)\prec z=x\prec(y\star z),\; (x\succ y)\prec z=x\succ(y\prec
z),\; (x\star y)\succ z=x\succ(y\succ z), x, y, z\in D,$$
where $\star:=\prec+\succ$. Similarly a {\bf tridendriform algebra}, defined by Loday and Ronco~\mcite{L-R1}, is a vector space $T$ with three binary operations $\prec, \succ$ and $\cdot$ that satisfy seven relations. Aguiar~\mcite{Ag3} showed that for a Rota-Baxter algebra $(R,P)$ of weight 0, the binary operations
$$ x \prec_Py:=xP(y), \quad x\succ_P y:=P(x)y, \quad \forall x, y\in R,$$
define a dendriform algebra on $R$. Similarly, Ebrahimi-Fard~\mcite{EF1} showed that, for a Rota-Baxter algebra $(R,P)$ of non-zero weight, the binary operations
$$ x \prec_Py:=xP(y), \quad x\succ_P y:=P(x)y, \quad x\cdot_P y:= \lambda xy, \quad \forall x, y\in R,$$
define a tridendriform algebra on $R$.

As an analogue of the tridendriform algebra, the concept of an {\bf NS algebra} was introduced by Leroux~\mcite{Le2}, to be a vector space $M$ with three binary
operations $\prec$, $\succ$ and $\bullet$ that satisfy four relations (see Eq.~(\mref{eq:ns})).
As an analogue of the Rota-Baxter algebra case, it was shown~\mcite{Le2} that, for a Nijenhuis algebra $(N,P)$, the binary operations
$$ x \prec_Py:=xP(y), \quad x\succ_P y:=P(x)y, \quad x\bullet_P y:= -P( xy), \quad \forall x, y\in R,$$
defines an NS algebra on $R$.

Considering the adjoint functor of the functor induced by the above mentioned map from Rota-Baxter algebras to (tri-)dendriform algebras, the Rota-Baxter universal enveloping algebra of a (tri-)dendriform algebra was constructed in~\mcite{EG2}. For this purpose, free Rota-Baxter algebras was first constructed.
\smallskip

In this paper we give a similar approach for Nijenhuis algebras, but we go beyond the case of Rota-Baxter algebras. Our first goal is to give an explicit construction of free Nijenhuis algebras in Section~\mref{sec:nonua}.
We consider both the cases when the free Nijenhuis algebra is generated by a set and by another algebra. Other than its role in the theoretical study of Nijenhuis algebras, this construction allows us to construct the universal enveloping algebra of an NS algebra. We achieve this in Section~\mref{sec:adj}.

Knowing that a Nijenhuis algebra gives an NS algebra, it is natural to ask what other dendriform type algebras that Nijenhuis algebras can give in a similar way. As a second application of our construction of free Nijenhuis algebras, we determine all ``quadratic nonsymmetric" relations that can be derived from Nijenhuis algebras and find that one can actually derive more relations than defined by the NS algebra in Eq.~(\mref{eq:ns}). This discussion is presented in Section~\mref{sec:sdn}. \smallskip

\noindent
{\bf Notation:} In this paper $\bfk$ is taken to be a field. A $\bfk$-algebra is taken to be nonunitary associative unless otherwise stated.


\section{Free Nijenhuis algebra on an algebra}
\mlabel{sec:nonua}
We start with the definition of free Nijenhuis algebras.
\begin{defn}
{\rm
Let $A$ be a $\bf k$-algebra. A free Nijenhuis
algebra over $A$ is a Nijenhuis algebra $\FN(A)$ with a
Nijenhuis operator $P_A$ and an algebra homomorphism $j_A:
A\to \FN(A)$ such that, for any  Nijenhuis algebra $N$ and
any  algebra homomorphism $f:A\to N$, there is a unique
 Nijenhuis algebra homomorphism $\free{f}: \FN(A)\to N$
such that $\free{f}\circ j_A=f$:
$$ \xymatrix{ A \ar[rr]^{j_A}\ar[drr]^{f} && \FN(A) \ar[d]_{\free{f}} \\
&& N}
$$
}
\mlabel{de:fna}
\end{defn}

For the construction of free Nijenhuis algebras, we follow the construction of free Rota-Baxter algebras~\mcite{EG2,Gub} by bracketed words. Alternatively, one can follow~\mcite{EG3} to give the construction by rooted trees that is more in the spirit of operads~\mcite{LV}. One can also follow the approach of Gr\"obner-Shirshov bases~\mcite{BCQ}.
Because of the lack of a uniform approach (see~\mcite{GSZ,GSZ2} for some recent attempts in this direction) and to be notationally self contained, we give some details.
We first display a $\bfk$-basis of the free Nijenhuis algebra in
terms of bracketed words in \S~\mref{ss:base}. The product on the free
Nijenhuis algebra is given in \S~\mref{ss:prodao} and the universal
property of the free Nijenhuis algebra is proved in \S~\mref{ss:proof}.

\subsection{A basis of the free Nijenhuis algebra}
\mlabel{ss:base} Let $A$ be a  $\bfk$-algebra with a
$\bfk$-basis $X$. We first display a $\bfk$-basis $\frakX_\infty$ of
$\FN(A)$ in terms of bracketed words from the alphabet set $X$.

Let $\lc$ and $\rc$ be symbols, called brackets, and let $X'=X\cup
\{\lc,\rc\}$. Let $M(X')$ denote the free semigroup generated by $X'$.

\begin{defn} (\mcite{EG2,Gub})
{\rm
Let $Y,Z$ be two subsets of $M(X')$. Define the {\bf alternating
product} of $Y$ and $Z$ to be \allowdisplaybreaks{
\begin{eqnarray}
\altx(Y,Z)&=&\Big( \bigcup_{r\geq 1} \big (Y\lc Z\rc \big)^r \Big)
\bigcup
    \Big(\bigcup_{r\geq 0} \big (Y\lc Z\rc \big)^r  Y\Big)
    \bigcup \Big( \bigcup_{r\geq 1} \big( \lc Z\rc Y \big )^r \Big)
 \bigcup \Big( \bigcup_{r\geq 0} \big (\lc Z\rc Y\big )^r \lc Z\rc \Big).
\mlabel{eq:wordsao}
\end{eqnarray}}
}
\mlabel{de:alt}
\end{defn}

We construct a sequence $\frakX_n$ of subsets of $M(X')$ by the
following recursion. Let $\frakX_0=X$ and, for $n\geq 0$, define
\allowdisplaybreaks{
\begin{equation}
\frakX_{n+1}=\altx(X,\frakX_n). \notag
\end{equation}
Further, define
\begin{eqnarray}
\frakX_\infty &=& \bigcup_{n\geq 0} \frakX_n = \dirlim \frakX_n.
\mlabel{eq:x3ao}
\end{eqnarray}}
Here the second equation in Eq. (\mref{eq:x3ao}) follows since
$\frakX_1\supseteq \frakX_0$ and, assuming $\frakX_n\supseteq
\frakX_{n-1}$, we have
$$\frakX_{n+1}=\altx(X,\frakX_n) \supseteq \altx(X,\frakX_{n-1})
    \supseteq \frakX_n.$$

By~\mcite{EG2,Gub} we have the disjoint union
\allowdisplaybreaks{
\begin{eqnarray}
\frakX_\infty & =&
    \Big( \bigsqcup_{r\geq 1} \big (X\lc \frakX_{\infty}\rc\big )^r \Big) \bigsqcup
    \Big(\bigsqcup_{r\geq 0} \big (X\lc \frakX_{\infty}\rc\big )^r  X\Big) \notag\\
 && \bigsqcup \Big( \bigsqcup_{r\geq 1} \big (\lc \frakX_{\infty}\rc X\big )^r \Big)
    \bigsqcup \Big( \bigsqcup_{r\geq 0} \big( \lc \frakX_{\infty}\rc X \big)^r
    \lc \frakX_{\infty}\rc \Big).
\mlabel{eq:words3}
\end{eqnarray}}
Further, every $\frakx\in \frakX_\infty$ has a unique decomposition
\begin{equation}
 \frakx=\frakx_1 \cdots  \frakx_b,
\mlabel{eq:st}
\end{equation}
where $\frakx_i$, $1\leq i\leq b$, is alternatively in $X$ or in
$\lc \frakX_\infty\rc$. This decomposition will be called the {\bf
standard decomposition} of $\frakx$.

For $\frakx$ in ${\frakX}_\infty$ with standard decomposition
$\frakx_1 \cdots  \frakx_b$, we define $b$ to be the {\bf breadth}
$b(\frakx)$ of $\frakx$, we define the {\bf head} $h(\frakx)$ of
$\frakx$ to be 0 (resp. 1) if $\frakx_1$ is in $X$ (resp. in $\lc
\frakX_\infty \rc$). Similarly define the {\bf tail} $t(\frakx)$ of
$\frakx$ to be 0 (resp. 1) if $\frakx_b$ is in $X$ (resp. in $\lc
\frakX_\infty \rc$).

\subsection{The product in a free Nijenhuis
algebra} \mlabel{ss:prodao} Let
$$\FN(A)=\bigoplus_{\frakx\in \frakX_\infty} \bfk \frakx.$$
We now define a product $\shpr$ on $\FN(A)$ by defining $\frakx\shpr
\frakx'\in \FN(A)$ for $\frakx,\frakx'\in \frakX_\infty$ and then
extending bilinearly. Roughly speaking, the product of $\frakx$ and
$\frakx'$ is defined to be the concatenation whenever $t(\frakx)\neq
h(\frakx')$. When $t(\frakx)=h(\frakx')$, the product is defined by
the product in $A$ or by the Nijenhuis relation in
Eq.~(\mref{eq:Nij}).

To be precise, we use induction on the sum $n:=d(\frakx)+d(\frakx')$
of the depths of $\frakx$ and $\frakx'$. Then $n\geq 0$. If $n=0$,
then $\frakx,\frakx'$ are in $X$ and so are in $A$ and we define
$\frakx\shpr \frakx'=\frakx \spr \frakx'\in A \subseteq \FN(A)$. Here
$\spr$ is the product in $A$.

Suppose $\frakx\shpr \frakx'$ have been defined for all
$\frakx,\frakx'\in \frakX_\infty$ with $n\geq k\geq 0$ and let
$\frakx, \frakx'\in \frakX_\infty$ with $n=k+1$.

First assume the breadth $b(\frakx)=b(\frakx')=1$. Then $\frakx$ and
$\frakx'$ are in $X$ or $\lc \frakX_\infty\rc$. Since $n=k+1$ is at least one, $\frakx$ and $\frakx'$ cannot be both in $X$.
We accordingly define
\begin{equation}
\frakx\shpr \frakx'=\left \{ \begin{array}{ll}
\frakx \frakx', & {\rm if\ } \frakx\in X, \frakx'\in \lc \frakX_\infty\rc,\\
\frakx \frakx', & {\rm if\ } \frakx\in \lc \frakX_\infty\rc, \frakx'\in X,\\
\lc \lc \ox\rc \shpr \ox'\rc +\lc \ox \shpr \lc \ox'\rc \rc -\lc \lc
\ox \shpr \ox' \rc\rc, & {\rm if\ } \frakx=\lc \ox\rc, \frakx'=\lc
\ox'\rc \in \lc \frakX_\infty \rc.
\end{array} \right .
\mlabel{eq:shprod0}
\end{equation}
Here the product in the first and second case are by concatenation and in the third case is
by the induction hypothesis since for the three products on the
right hand side we have
\begin{eqnarray*}
d(\lc\ox \rc)+ d(\ox') &=& d(\lc \ox \rc)+d(\lc \ox' \rc)-1
= d(\frakx)+d(\frakx')-1,\\
d(\ox)+d(\lc \ox'\rc) &=& d(\lc \ox \rc)+d(\lc \ox'\rc)-1
= d(\frakx)+ d(\frakx')-1,\\
d(\ox)+ d(\ox') &=& d(\lc \ox \rc)-1+ d(\lc \ox' \rc)-1 =
d(\frakx)+d(\frakx')-2
\end{eqnarray*}
which are all less than or equal to $k$.

Now assume $b(\frakx)>1$ or $b(\frakx')>1$. Let
$\frakx=\frakx_1\cdots\frakx_b$ and
$\frakx'=\frakx'_1\cdots\frakx'_{b'}$ be the standard decompositions
from Eq.~(\mref{eq:st}). We then define
\begin{equation}
\frakx \shpr \frakx'= \frakx_1\cdots \frakx_{b-1}(\frakx_b\shpr
\frakx'_1)\,
    \frakx'_{2}\cdots \frakx'_{b'}
\end{equation}
where $\frakx_b\shpr \frakx'_1$ is defined by
Eq.~(\mref{eq:shprod0}) and the rest is given by concatenation. The concatenation is well-defined since by Eq.~(\mref{eq:shprod0}), we
have $h(\frakx_b)=h(\frakx_b\shpr \frakx'_1)$ and
$t(\frakx'_1)=t(\frakx_b\shpr \frakx'_1)$. Therefore,
$t(\frakx_{b-1})\neq h(\frakx_b\shpr \frakx'_1)$ and
$h(\frakx'_2)\neq t(\frakx_b\shpr \frakx'_1)$.
\medskip

We have the following simple properties of $\shpr$.
\begin{lemma} Let $\frakx,\frakx'\in \frakX_\infty$. We have the following
statements.
\begin{enumerate}
\item $h(\frakx)=h(\frakx\shpr \frakx')$ and $t(\frakx')=t(\frakx\shpr \frakx')$.
\mlabel{it:mat0}
\item If $t(\frakx)\neq h(\frakx')$, then
$\frakx \shpr \frakx' =\frakx \frakx'$ (concatenation).
\mlabel{it:mat1}
\item If $t(\frakx)\neq h(\frakx')$, then for any $\frakx''\in \frakX_\infty$,
$$(\frakx\frakx')\shpr \frakx'' =\frakx(\frakx' \shpr \frakx''), \quad
\frakx''\shpr (\frakx \frakx') =(\frakx'' \shpr \frakx) \frakx'.$$
\mlabel{it:mat2}
\end{enumerate}
\mlabel{lem:match}
\end{lemma}

Extending $\shpr$ bilinearly, we obtain a binary operation
$$ \FN(A)\otimes \FN(A) \to \FN(A).$$
For $\frakx\in \frakX_\infty$, define
\begin{equation}
N_A(\frakx)=\lc \frakx \rc. \mlabel{eq:RBop}
\end{equation}
Obviously $\lc \frakx \rc$ is again in $\frakX_\infty$. Thus $N_A$
extends to a linear operator $N_A$ on $\FN(A)$. Let
$$j_X:X\to \frakX_\infty \to \FN(A)$$
be the natural injection which extends to an algebra injection
\begin{equation}
j_A: A \to \FN(A). \mlabel{eq:jo}
\end{equation}

The following is our first main result which will be proved in the
next subsection.
\begin{theorem}
Let $A$ be a  $\bfk$-algebra with a $\bfk$-basis $X$.
\begin{enumerate}
\item
The pair $(\FN(A),\shpr)$ is an algebra.
\mlabel{it:alg}
\item
The triple $(\FN(A),\shpr,N_A)$ is a Nijenhuis algebra.
\mlabel{it:RB}
\item
The quadruple $(\FN(A),\shpr,N_A,j_A)$ is the free
Nijenhuis algebra on the algebra $A$. \mlabel{it:free}
\end{enumerate}
\mlabel{thm:freeao}
\end{theorem}

The following corollary of the theorem will be used later in the
paper.
\begin{coro}
Let $M$ be a $\bfk$-module and let $T(M)=\bigoplus_{n\geq 1} M^{\ot
n}$ be the reduced tensor algebra over $M$. Then $\FN(T(M))$, together with
the natural injection $i_M: M\to T(M) \xrightarrow{j_{T(M)}}
\FN(T(M))$, is a free  Nijenhuis algebra over $M$, in the
sense that, for any  Nijenhuis algebra $N$ and
$\bfk$-module map $f: M\to N$ there is a unique  Nijenhuis
algebra homomorphism $\freev{f}: \FN(T(M)) \to N$ such that
$\freev{f} \circ k_M = f$. \mlabel{co:vecfree}
\end{coro}
\begin{proof}
This follows immediately from Theorem~\mref{thm:freeao} and the fact that the construction of the free algebra on a module (resp. free Nijenhuis algebra on an algebra, resp. free Nijenhuis on a module) is the left adjoint functor of the forgetful functor from algebras to modules (resp. from Nijenhuis algebras to algebras, resp. from Nijenhuis algebras to modules), and the fact that the composition of two left adjoint functors is the left adjoint functor of the composition.
\end{proof}

\subsection{The proof of Theorem~\mref{thm:freeao}}
\mlabel{ss:proof}
\begin{proof}
\mref{it:alg}. We just need to verify the associativity. For this we
only need to verify
\begin{equation}
 (\frakx'\shpr \frakx'')\shpr \frakx''' =\frakx'\shpr(\frakx'' \shpr \frakx''')
\mlabel{eq:assx}
\end{equation}
for $\frakx',\frakx'',\frakx'''\in \frakX_\infty$. We will do this
by induction on the sum of the depths
$n:=d(\frakx')+d(\frakx'')+d(\frakx''')$. If $n=0$, then all of
$\frakx',\frakx'',\frakx'''$ have depth zero and so are in $X$. In
this case the product $\shpr$ is given by the product $\spr$ in $A$
and so is associative.

Assume the associativity holds for $n\leq k$ and assume that
$\frakx',\frakx'',\frakx'''\in \frakX_\infty$ have
$n=d(\frakx')+d(\frakx'')+d(\frakx''')=k+1.$

If $t(\frakx')\neq h(\frakx'')$, then by Lemma~\mref{lem:match},
$$ (\frakx' \shpr \frakx'') \shpr \frakx'''=(\frakx'\frakx'')\shpr \frakx'''
= \frakx' (\frakx'' \shpr \frakx''') =\frakx'\shpr (\frakx''\shpr
\frakx''').$$
A similar argument holds when $t(\frakx'')\neq h(\frakx''')$.

Thus we only need to verify the associativity when
$t(\frakx')=h(\frakx'')$ and $t(\frakx'')=h(\frakx''')$. We next
reduce the breadths of the words.

\begin{lemma}
If the associativity
$$(\frakx' \shpr \frakx'')\shpr \frakx'''=
 \frakx'\shpr (\frakx'' \shpr \frakx''') $$
holds for all $\frakx', \frakx''$ and $\frakx'''$ in $\frakX_\infty$
of breadth one, then it holds for all $\frakx', \frakx''$ and
$\frakx'''$ in $\frakX_\infty$. \mlabel{lem:ell}
\end{lemma}

\begin{proof}
We use induction on the sum of breadths
$m:=b(\frakx')+b(\frakx'')+b(\frakx''')$. Then $m\geq 3$. The case
when $m=3$ is the assumption of the lemma. Assume the associativity
holds for $3\leq m \leq j$ and take $\frakx', \frakx'',\frakx'''\in
\frakX_\infty$ with $m = j+1.$ Then $j+1\geq 4$. So at least one of
$\frakx',\frakx'',\frakx'''$ have breadth greater than or equal to
2.

First assume $b(\frakx')\geq 2$. Then $\frakx'=\frakx'_1\frakx'_2$
with $\frakx'_1,\, \frakx'_2\in \frakX_\infty$ and $t(\frakx'_1)\neq
h(\frakx'_2)$. Thus by Lemma~\mref{lem:match}, we obtain
$$
 (\frakx'\shpr \frakx'') \shpr \frakx'''=
((\frakx'_1\frakx'_2)\shpr \frakx'')\shpr \frakx'''
= (\frakx'_1 (\frakx'_2 \shpr \frakx''))\shpr \frakx'''
= \frakx'_1 ((\frakx'_2 \shpr \frakx'') \shpr \frakx'''). $$
Similarly,
$$ \frakx'\shpr (\frakx'' \shpr \frakx''')=
(\frakx'_1\frakx'_2)\shpr (\frakx''\shpr \frakx''')
= \frakx'_1 (\frakx'_2 \shpr (\frakx''\shpr \frakx''')).
$$
Thus
$ (\frakx'\shpr \frakx'') \shpr \frakx'''=
  \frakx'\shpr (\frakx'' \shpr \frakx''')$
whenever
$ (\frakx'_2 \shpr \frakx'') \shpr \frakx'''=
\frakx'_2 \shpr (\frakx''\shpr \frakx''').$
The latter follows from the induction hypothesis.
A similar proof works if $b(\frakx''')\geq 2.$

Finally if $b(\frakx'')\geq 2$, then $\frakx''=\frakx''_1\frakx''_2$
with $\frakx''_1,\,\frakx''_2\in \frakX_\infty$ and
$t(\frakx''_1)\neq h(\frakx''_2)$. By applying Lemma~\mref{lem:match}
repeatedly, we obtain

$$
(\frakx' \shpr \frakx'')\shpr \frakx'''=
(\frakx' \shpr (\frakx''_1 \frakx''_2)) \shpr \frakx''' \\
= ((\frakx' \shpr \frakx''_1)\frakx''_2)\shpr \frakx'''
= (\frakx'\shpr \frakx''_1)(\frakx''_2 \shpr \frakx''').
$$
In the same way, we have
$$(\frakx'\shpr \frakx''_1)(\frakx''_2 \shpr \frakx''')
= \frakx'\shpr (\frakx'' \shpr \frakx''').$$ This again proves the
associativity.
\end{proof}

To summarize, our proof of the associativity has been reduced to the
special case when $\frakx',\frakx'',\frakx'''\in \frakX_\infty$ are
chosen so that
\begin{enumerate}
\item
$n:= d(\frakx')+d(\frakx'')+d(\frakx''')=k+1\geq 1$ with the
assumption that the associativity holds when $n\leq k$.
\mlabel{it:sp1}
\item
the elements have breadth one and \mlabel{it:sp2}
\item
$t(\frakx')=h(\frakx'')$ and $t(\frakx'')=h(\frakx''')$.
\mlabel{it:sp3}
\end{enumerate}
By item \mref{it:sp2}, the head and tail of each of the elements are
the same. Therefore by item \mref{it:sp3}, either all the three
elements are in $X$ or they are all in $\lc \frakX_\infty \rc$. If
all of $\frakx',\frakx'',\frakx'''$ are in $X$, then as already
shown, the associativity follows from the associativity in $A$.

So it remains to consider the case when $\frakx',\frakx'',\frakx'''$ are all in $\lc
\frakX_\infty \rc$. Then $\frakx'=\lc \ox'\rc, \frakx''=\lc \ox''
\rc, \frakx'''=\lc \ox'''\rc$ with $\ox',\ox'',\ox'''\in
\frakX_\infty$. Using Eq.~(\mref{eq:shprod0}) and bilinearity of the
product $\shpr$, we have
\allowdisplaybreaks{
\begin{eqnarray*}
(\frakx'\shpr \frakx'')\shpr \frakx'''
&=& \big( \lc \lc \ox'\rc \shpr
\ox ''\rc +\lc\ox'\shpr \lc\ox''\rc\rc
    -\lc\lc\ox'\shpr \ox''\rc \rc\big ) \shpr \lc \ox'''\rc \\
&=& \lc\lc \ox'\rc \shpr \ox''\rc \shpr \lc\ox'''\rc
    + \lc\ox'\shpr \lc \ox''\rc \rc\shpr \lc \ox'''\rc
    -\lc \lc \ox'\shpr \ox''\rc\rc \shpr \lc\ox'''\rc \\
&=&  \lc\lc\lc \ox'\rc\shpr \ox''\rc\shpr \ox''' \rc
    + \lc\big(\lc\ox'\rc \shpr \ox''\big) \shpr \lc\ox'''\rc\rc
    -\lc \lc\big(\lc\ox'\rc \shpr\ox''\big)\shpr \ox'''\rc\rc\\
&& + \lc\lc\ox'\shpr\lc\ox''\rc\rc \shpr \ox'''\rc
    + \lc\big(\ox'\shpr\lc \ox''\rc\big) \shpr\lc \ox'''\rc\rc
    -\lc \lc\big(\ox'\shpr \lc \ox''\rc \big) \shpr \ox'''\rc\rc \\
&& - \lc \lc \lc \ox'\shpr \ox''\rc\rc\shpr \ox'''\rc
    -\lc \lc \ox'\shpr \ox''\rc\shpr \lc \ox'''\rc \rc
    +\lc\lc \lc \ox'\shpr \ox''\rc \shpr \ox'''\rc\rc.
\end{eqnarray*}}
Applying the induction hypothesis in $n$ to the fifth term $\big
(\ox'\shpr\lc \ox''\rc\big) \shpr\lc \ox'''\rc$ and the eighth term, and then use
Eq.~(\mref{eq:shprod0}) again, we obtain
\allowdisplaybreaks{
\begin{eqnarray*}
(\frakx'\shpr \frakx'')\shpr \frakx'''
&=& \lc\lc\lc \ox'\rc\shpr \ox''\rc\shpr \ox''' \rc
    + \lc\big(\lc\ox'\rc \shpr \ox''\big) \shpr \lc\ox'''\rc\rc
    -\lc \lc\big(\lc\ox'\rc \shpr\ox''\big)\shpr \ox'''\rc\rc\\
&& + \lc\lc\ox'\shpr\lc\ox''\rc\rc \shpr \ox'''\rc
    + \lc\ox'\shpr\lc \lc \ox''\rc \shpr\ox'''\rc\rc
    + \lc\ox'\shpr\lc \ox'' \shpr\lc \ox'''\rc\rc\rc\\
&&  - \lc\ox'\shpr\lc \lc\ox'' \shpr \ox'''\rc\rc\rc
    -\lc \lc\big(\ox'\shpr \lc \ox''\rc \big) \shpr \ox'''\rc\rc
    - \lc \lc \lc \ox'\shpr \ox''\rc\rc\shpr \ox'''\rc
    \\
&&-\lc \lc  \lc\ox'\shpr \ox''\rc\shpr \ox'''\rc \rc
  -\lc \lc (\ox'\shpr \ox'')\shpr \lc \ox'''\rc \rc\rc
  +\lc \lc \lc(\ox'\shpr \ox'')\shpr  \ox'''\rc\rc \rc\\
&&  +\lc\lc \lc \ox'\shpr \ox''\rc \shpr \ox'''\rc\rc\\
&=& \lc\lc\lc \ox'\rc\shpr \ox''\rc\shpr \ox''' \rc
    + \lc\big(\lc\ox'\rc \shpr \ox''\big) \shpr \lc\ox'''\rc\rc
    -\lc \lc\big(\lc\ox'\rc \shpr\ox''\big)\shpr \ox'''\rc\rc\\
&& + \lc\lc\ox'\shpr\lc\ox''\rc\rc \shpr \ox'''\rc
    + \lc\ox'\shpr\lc \lc \ox''\rc \shpr\ox'''\rc\rc
    + \lc\ox'\shpr\lc \ox'' \shpr\lc \ox'''\rc\rc\rc\\
&&  - \lc\ox'\shpr\lc \lc\ox'' \shpr \ox'''\rc\rc\rc
    -\lc \lc\big(\ox'\shpr \lc \ox''\rc \big) \shpr \ox'''\rc\rc
    - \lc \lc \lc \ox'\shpr \ox''\rc\rc\shpr \ox'''\rc
    \\
&&  -\lc \lc (\ox'\shpr \ox'')\shpr \lc \ox'''\rc \rc\rc
  +\lc \lc \lc(\ox'\shpr \ox'')\shpr  \ox'''\rc\rc \rc.
\end{eqnarray*}}
By a similar computation, we obtain
\allowdisplaybreaks{
\begin{eqnarray*}
\frakx' \shpr \big(\frakx''\shpr \frakx'''\big)
&=& \lc\lc\lc\ox'\rc\shpr \ox''\rc\shpr \ox'''\rc
    + \lc \lc \ox'\shpr \lc\ox''\rc\rc \shpr \ox'''\rc
    -\lc \lc\lc\ox'\shpr\ox''\rc\rc\shpr\ox'''\rc\\
&&  +\lc \ox'\shpr \lc \lc \ox''\rc \shpr \ox'''\rc\rc
    - \lc\lc \ox'\shpr \big(\lc\ox''\rc\shpr \ox'''\big)\rc\\
&&  + \lc\lc \ox'\rc\shpr \big(\ox''\shpr \lc \ox'''\rc \big) \rc
    + \lc \ox' \shpr \lc \ox'' \shpr \lc \ox'''\rc\rc\rc
    -\lc \lc \ox'\shpr \big( \ox''\shpr \lc \ox'''\rc \big)\rc\rc\\
&&  - \lc\lc\lc\ox'\rc\shpr (\ox''\shpr \ox''') \rc\rc
    +\lc\lc\lc\ox'\shpr (\ox''\shpr \ox''') \rc\rc\rc
    -  \lc \ox'\shpr \lc\lc \ox''\shpr \ox'''\rc\rc\rc.
\end{eqnarray*}}
Now by induction, the $i$-th term in the expansion of $(\frakx'\shpr
\frakx'')\shpr \frakx'''$ matches with the $\sigma(i)$-th term  in
the expansion of $\frakx'\shpr(\frakx'' \shpr \frakx''')$. Here the
permutation $\sigma\in \Sigma_{11}$ is given by
\begin{equation}
\sigma= \left ( \begin{array}{ccccccccccc} 1&2&3&4&5&6&7&8&9&10&11\\
    1&6&9&2&4&7&11&5&3&8&10\end{array} \right ).
\mlabel{eq:sigma}
\end{equation}
This completes the proof of Theorem~\mref{thm:freeao}.\mref{it:alg}.

\mref{it:RB}. The proof follows from the definition
$N_A(\frakx)=\lc \frakx\rc$ and Eq. (\mref{eq:shprod0}).

\mref{it:free}. Let $(N,\ast,P)$ be a Nijenhuis algebra with multiplication $\ast$. Let $f:A\to N$ be a  $\bfk$-algebra homomorphism. We will construct a
$\bfk$-linear map $\free{f}:\FN(A) \to N$ by defining
$\free{f}(\frakx)$ for $\frakx\in \frakX_\infty$. We achieve this by
defining $\free{f}(\frakx)$ for $\frakx\in \frakX_n,\ n\geq 0$,
inductively on $n$. For $\frakx\in \frakX_0:=X$, define
$\free{f}(\frakx)=f(\frakx).$ Suppose $\free{f}(\frakx)$ has been
defined for $\frakx\in \frakX_n$ and consider $\frakx$ in
$\frakX_{n+1}$ which is, by definition and Eq.~(\mref{eq:words3}),
\allowdisplaybreaks{
\begin{eqnarray*} \altx(X,\frakX_{n})& =&
    \Big( \bigsqcup_{r\geq 1} (X\lc \frakX_{n}\rc)^r \Big) \bigsqcup
    \Big(\bigsqcup_{r\geq 0} (X\lc \frakX_{n}\rc)^r  X\Big) \\
 &&    \bigsqcup \Big( \bigsqcup_{r\geq 0} \lc \frakX_{n}\rc (X\lc \frakX_{n}\rc)^r \Big)
   \bigsqcup \Big( \bigsqcup_{r\geq 0} \lc \frakX_{n}\rc (X\lc \frakX_{n}\rc)^r X\Big).
\end{eqnarray*}}
Let $\frakx$ be in the first union component $\bigsqcup_{r\geq 1}
(X\lc \frakX_{n}\rc)^r$ above. Then
$$\frakx = \prod_{i=1}^r(\frakx_{2i-1} \lc \frakx_{2i} \rc)$$
for $\frakx_{2i-1}\in X$ and $\frakx_{2i}\in \frakX_n$, $1\leq i\leq
r$. By the construction of the multiplication $\shpr$ and the
Nijenhuis operator $N_A$, we have
$$\frakx= \shpr_{i=1}^r(\frakx_{2i-1} \shpr \lc \frakx_{2i}\rc)
    = \shpr_{i=1}^r(\frakx_{2i-1} \shpr N_A(\frakx_{2i})).$$
Define
\begin{equation}
\free{f}(\frakx) = \ast_{i=1}^r \big(\free{f}(\frakx_{2i-1})
    \ast N\big (\free{f}(\frakx_{2i})) \big).
\mlabel{eq:hom}
\end{equation}
where the right hand side is well-defined by the induction
hypothesis. Similarly define $\free{f}(\frakx)$ if $\frakx$ is in
the other union components. For any $\frakx\in \frakX_\infty$, we
have $P_A(\frakx)=\lc \frakx\rc\in \frakX_\infty$, and by the definition of $\free{f}$ in (Eq. (\mref{eq:hom})), we have
\begin{equation}
\free{f}(\lc \frakx \rc)=P(\free{f}(\frakx)). \mlabel{eq:hom1-2}
\end{equation}
So $\free{f}$ commutes with the Nijenhuis operators. Combining this
equation with Eq.~(\mref{eq:hom}) we see that if
$\frakx=\frakx_1\cdots \frakx_b$ is the standard decomposition of
$\frakx$, then
\begin{equation}
 \free{f}(\frakx)=\free{f}(\frakx_1)*\cdots * \free{f}(\frakx_b).
\mlabel{eq:staohom}
\end{equation}

Note that this is the only possible way to define $\free{f}(\frakx)$
in order for $\free{f}$ to be a Nijenhuis algebra homomorphism
extending $f$.

It remains to prove that the map $\free{f}$ defined in
Eq.~(\mref{eq:hom}) is indeed an algebra homomorphism. For this we
only need to check the multiplicity
\begin{equation}
\free{f} (\frakx \shpr \frakx')=\free{f}(\frakx) \ast
\free{f}(\frakx') \mlabel{eq:hom2}
\end{equation}
for all $\frakx,\frakx'\in \frakX_\infty$. For this we use induction
on the sum of depths $n:=d(\frakx)+d(\frakx')$. Then $n\geq 0$. When
$n=0$, we have $\frakx,\frakx'\in X$. Then Eq.~(\mref{eq:hom2})
follows from the multiplicity of $f$. Assume the multiplicity holds
for $\frakx,\frakx' \in \frakX_\infty$ with $n\geq k$ and take
$\frakx,\frakx'\in \frakX_\infty$ with $n=k+1$. Let
$\frakx=\frakx_1\cdots \frakx_b$ and
$\frakx'=\frakx'_1\cdots\frakx'_{b'}$ be the standard
decompositions. Since $n=k+1\geq 1$, at least one of $\frakx_b$ and $\frakx'_{b'}$ is in $\lc \frakX_\infty\rc$. Then by Eq.~(\mref{eq:shprod0}) we have,
\begin{align*}
\free{f}(\frakx_b\shpr \frakx'_1)&= \left \{\begin{array}{ll}
\free{f}(\frakx_b \frakx'_1), & {\rm if\ } \frakx_b\in X, \frakx'_1\in \lc \frakX_\infty\rc,\\
\free{f}(\frakx_b \frakx'_1), & {\rm if\ } \frakx_b\in \lc
\frakX_\infty\rc,
    \frakx'_1\in X,\\
\free{f}\big( \lc \lc \ox_b\rc \shpr \ox'_1\rc +\lc \ox_b \shpr \lc
\ox'_1\rc \rc -\lc \lc \ox_b \shpr \ox'_1 \rc\rc\big), & {\rm if\ }
\frakx_b=\lc \ox_b\rc, \frakx'_1=\lc \ox'_1\rc \in \lc \frakX_\infty
\rc.
\end{array} \right .
\end{align*}
In the first two cases, the right hand side is
$\free{f}(\frakx_b)*\free{f}(\frakx'_1)$ by the definition of
$\free{f}$. In the third case, we have, by Eq.~(\mref{eq:hom1-2}),
the induction hypothesis and the Nijenhuis relation of the operator $P$ on $N$,
\begin{align*}
&\free{f}\big( \lc \lc \ox_b\rc \shpr \ox'_1\rc
    +\lc \ox_b \shpr \lc \ox'_1\rc \rc
-\lc \lc \ox_b \shpr \ox'_1 \rc\rc\big)\\
=&\free{f}(\lc \lc \ox_b\rc \shpr \ox'_1\rc) + \free{f}(\lc \ox_b
\shpr \lc \ox'_1\rc \rc)
-\free{f}(\lc \lc \ox_b \shpr \ox'_1 \rc\rc)\\
=&P(\free{f}(\lc \ox_b\rc \shpr \ox'_1)) + P(\free{f}(\ox_b \shpr
\lc \ox'_1\rc ))
-P(\free{f}(\lc\ox_b \shpr \ox'_1 \rc))\\
=&P(\free{f}(\lc \ox_b\rc)*\free{f}(\ox'_1)) + P(\free{f}(\ox_b)
*\free{f}( \lc \ox'_1\rc ))
-P(P(\free{f}(\ox_b) * \free{f}(\ox'_1)) )\\
=&P(P(\free{f}(\ox_b))*\free{f}(\ox'_1)) + P(\free{f}(\ox_b)
*P(\free{f}(\ox'_1)))
-P(P((\free{f}(\ox_b) * \free{f}(\ox'_1)) )\\
=& P(\free{f}(\ox_b))*P(\free{f}(\ox'_1))\\
=& \free{f}(\lc \ox_b\rc) * \free{f}(\lc\ox'_1\rc)\\
=& \free{f} (\frakx_b) *\free{f}(\frakx'_1).
\end{align*}
Therefore $\free{f}(\frakx_b\shpr
\frakx'_1)=\free{f}(\frakx_b)*\free{f}(\frakx'_1)$. Then
\begin{align*}
\free{f}(\frakx\shpr \frakx')&=
\free{f}\big(\frakx_1\cdots\frakx_{b-1}(\frakx_b\shpr
\frakx'_1)\frakx'_2\cdots
    \frakx'_{b'}\big) \\
&= \free{f}(\frakx_1)*\cdots *\free{f}(\frakx_{b-1})*
\free{f}(\frakx_b\shpr \frakx'_1)*\free{f}(\frakx'_2)\cdots
    \free{f}(\frakx'_{b'})\\
&= \free{f}(\frakx_1)*\cdots *\free{f}(\frakx_{b-1})*
\free{f}(\frakx_b)* \free{f} (\frakx'_1)*\free{f}(\frakx'_2)\cdots
    \free{f}(\frakx'_{b'})\\
&= \free{f}(\frakx)*\free{f}(\frakx').
\end{align*}
This is what we need.
\end{proof}


\section{NS algebras and their universal enveloping algebras}
\mlabel{sec:adj}

The concept of an NS algebra was introduced by Leroux~\mcite{Le2} as an analogue of the dendriform algebra of Loday~\mcite{Lo1} and the tridendriform algebra of Loday and Ronco~\mcite{L-R1}.
\begin{defn}
{\rm
An {\bf NS algebra} is a module $M$ with three binary
operations $\prec$, $\succ$ and $\bullet$ that satisfy the following
four relations
\begin{eqnarray}
&(x\prec y)\prec z=x\prec (y\star z),\quad
(x\succ y)\prec z=x\succ (y\prec z), & \notag \\
&(x\star y)\succ z=x\succ
(y\succ z),\quad
(x\star y)\bullet z+(x\bullet y)\prec z = x \succ (y\bullet z)
+x\bullet (y\star z).&
\mlabel{eq:ns}
\end{eqnarray}
for $x,y,z\in M$. Here $\star$ denotes $\prec+\succ+\,
\bullet$.
}
\end{defn}
NS algebras share similar properties as dendriform algebras. For example, the operation
$\star$ defines an associative operation. Another similarity is the following theorem which is an analogue of the results of Aguiar~\mcite{Ag3} and Ebrahimi-Fard~\mcite{EF1} that a Rota-Baxter algebra gives a dendriform algebra or a tridendriform algebra.

\begin{theorem} $($\mcite{Le2}$)$ A Nijenhuis algebra $(N,P)$  defines an NS algebra $(N,\prec_P,\succ_P,\bullet_P)$, where
\begin{equation} x\prec_P y=xP(y),\ x\succ_P y=P(x)y, x\bullet_P y=-P(xy).
\mlabel{eq:eqs}
\end{equation}
\mlabel{thm:le}
\end{theorem}
Let $\NA$ denote the category of Nijenhuis algebras and let $\NS$ denote the category of NS algebras.
It is easy to see
that the map from $\NA$ to $\NS$ in Theorem~\mref{thm:le} is compatible with the morphisms in the two categories. Thus we obtain a functor
\begin{equation}
\cale: \NA \to \NS.
\mlabel{eq:nsdn}
\end{equation}
We will study its left adjoint functor.

Motivated by the enveloping algebra of a Lie
algebra and the Rota-Baxter enveloping algebra of a tridendriform algebra~\mcite{EG2}, we are naturally led to the following definition.
\begin{defn}
{\rm
Let $M$ be an NS-algebra. A {\bf universal enveloping Nijenhuis algebra} of
 $M$ is a Nijenhuis algebra
$\UN(M)\in \NA$ with a homomorphism $\rho: M\to \UN(M)$ in $\NS$  such
that for any $N\in \NA$ and homomorphism $f:M\to N$ in $\NS$, there is a unique $\den{f}: \UN(M)\to N$ in $\NA$ such that $\den{f} \circ
\rho =f$.
}\mlabel{de:env}
\end{defn}

Let $M:=(M,\prec,\succ,\bullet)\in  \NS$.  Let $T(M)=\bigoplus_{n\geq 1}
M^{\ot n}$ be the tensor algebra. Then $T(M)$ is the free  algebra generated by the $\bfk$-module
$M$.
By Corollary~\mref{co:vecfree}, $\FN(T(M))$, with the natural
injection $i_M: M\to T(M) \to \FN(T(M))$, is the free Nijenhuis
algebra over the vector space $M$.

Let $J_M$ be the Nijenhuis ideal of $\FN(T(M))$ generated by the set
\begin{equation}
\big \{ x\prec y - xP( y),\;
    x\succ y - P( x) y,\;x\bullet y=P(x\otimes y)  \big|\ x,y\in M \big\}
\mlabel{eq:gendend}
\end{equation}
Let $\pi: \FN(T(M))\to \FN(T(M))/J_M$ be the quotient map.

\begin{theorem}
Let $(M,\prec,\succ,\bullet)$ be an NS algebra. The quotient Nijenhuis algebra $\FN(T(M))/J_M$, together with $\rho:=
\pi \circ i_M$, is the universal enveloping Nijenhuis algebra of $M$. \mlabel{thm:envdend}
\end{theorem}

\begin{proof}
The proof is similar to the case of tridendriform algebras and Rota-Baxter algebras~\mcite{EG2}. So we skip some of the details.

Let $(N,P)$ be a Nijenhuis algebra and let $f:M\to N$ be a homomorphism
in $\NS$. More precisely, we have $f:(M,\prec,\succ,\bullet) \to (N,\prec_P',\succ_P',\bullet_P)$. We will complete the following commutative
diagram, using notations from Corollary~\mref{co:vecfree}.
\begin{equation}
\xymatrix{ & T(M) \ar[rd]^{j_{T(M)}} \ar@{.>}[lddd]^{\freea{f}} & \\
M \ar[rr]^{i_M} \ar[dd]_f \ar[ru]^{k_M} && \FN(T(M)) \ar[dd]^\pi \ar@{.>}[ddll]_{\freev{f}} \\
&& \\
N && \FN(T(M))/J_M \ar@{.>}[ll]_{\den{f}} }
\end{equation}

By the universal property of the free algebra $T(M)$ over $M$, there
is a unique homomorphism $\freea{f}:T(M)\to N$  such that
$\freea{f}\circ k_M =f$. So $\freea{f}(x_1\ot \cdots \ot
x_n)=f(x_1) * \cdots * f(x_n)$. Here $*$ is the product in $N$. Then
by the universal property of the free Nijenhuis algebra $\FN(T(M))$
over $T(M)$, there is a unique morphism $\free{\freea{f}}:\FN(T(M))
\to N$ in $\NA$ such that $\free{\freea{f}}\circ j_{T(M)}
=\freea{f}$. By Corollary~\mref{co:vecfree},
$\free{\freea{f}}=\freev{f}.$ Then
\begin{equation} \freev{f}\circ i_M =\freev{f} \circ j_{T(M)} \circ
k_M = \freea{f} \circ k_M = f. \mlabel{eq:free2}
\end{equation}
So for any $x,y\in M$, we check that
\begin{eqnarray*}
\freev{f}(x\prec y - x P( y))=0, \quad
\freev{f}(x\succ y - P(x)y)=0, \quad
\freev{f}(x\bullet y -  P(x\otimes y))=0.
\end{eqnarray*}
Thus $J_M$ is in $\ker(\freev{f})$ and there is a morphism $\den{f}:
\FN(T(M))/J_M\to N$ in $\NA$ such that $\freev{f}=\den{f} \circ \pi$.
Then by the definition of $\rho=\pi \circ i_M$ in the theorem and
Eq. (\mref{eq:free2}), we have
$$ \den{f}\circ \rho = \den{f} \circ \pi \circ i_M=\freev{f}\circ i_M=f.$$
This proves the existence of $\den{f}$.

Suppose $\den{f}':\FN(T(M))/J_M \to N$ is also a homomorphism in $\NA$
such that $\den{f}'\circ \rho=f$. Then
$$ (\den{f}' \circ \pi)\circ i_M = f = (\den{f}\circ \pi)\circ i_M.$$
By Corollary~\mref{co:vecfree}, the free Nijenhuis algebra
$\FN(T(M))$ over the algebra $T(M)$ is also the free Nijenhuis
algebra over the vector space $M$ with respect the natural injection
$i_M$. So we have $\den{f}'\circ \pi = \den{f} \circ \pi$ in $\NA$.
Since $\pi$ is surjective, we have $\den{f}'=\den{f}$. This proves
the uniqueness of $\den{f}$.
\end{proof}


\section{From Nijenhuis algebras to N-dendriform algebras}
\mlabel{sec:sdn}

In this section, we consider an inverse of Theorem~\mref{thm:le} in the following sense.  Suppose $(N,P)$ is a Nijenhuis algebra and define binary operations $$
x\prec_P y=xP(y),\ x\succ_P y=P(x)y, x\bullet_P y=-P(xy).$$
By Theorem~\mref{thm:le}, the three operations satisfy the NS relations in Eq.~(\mref{eq:ns}). Our inverse question is, what other quadratic nonsymmetric relations could
$(N,\prec_P,\succ_P,\bullet_P)$ satisfy? We recall some background on binary quadratic nonsymmetric operads in order to make the question precise. We then determine all the quadratic nonsymmetric relations that are consistent with the Nijenhuis operator.

\subsection{Background and the statement of Theorem~\mref{thm:wdn}}
For details on binary quadratic nonsymmetric operads, see~\mcite{Gub,LV}.

\begin{defn}
{\rm
Let $\bfk$ be a field.
\begin{enumerate}
\item
A {\bf graded vector space} is a sequence  $\calp:=\{\calp_n\}_{n\geq 0}$ of $\bfk$-vector spaces $\calp_n, n\geq 0$.
\item
A {\bf nonsymmetric (ns) operad} is a graded vector space $\calp=\{\calp_n\}_{n\geq 0}$ equipped with {\bf partial compositions}:
\begin{equation}
\circ_i:=\circ_{m,n,i}: \calp_m\ot \calp_n\longrightarrow \calp_{m+n-1}, \quad 1\leq i\leq m,
\mlabel{eq:opc}
\end{equation}
such that, for $\lambda\in\calp_\ell, \mu\in\calp_m$ and $\nu\in\calp_n$, the following relations hold.
\begin{enumerate}
\item[(i)] $
(\lambda \circ_i \mu)\circ_{i-1+j}\nu = \lambda\circ_i (\mu\circ_j\nu), \quad 1\leq i\leq \ell, 1\leq j\leq m.$
\mlabel{it:esc}
\item[(ii)]$(\lambda\circ_i\mu)\circ_{k-1+m}\nu =(\lambda\circ_k\nu)\circ_i\mu, \quad
1\leq i<k\leq \ell.$
\mlabel{it:epc}
\item[(iii)]
There is an element $\id\in \calp_1$ such that $\id\circ \mu=\mu$ and $\mu\circ\id=\mu$ for $\mu\in \calp_n, n\geq 0$.
\mlabel{it:id}
\end{enumerate}
\end{enumerate}
}
\end{defn}

An ns operad $\calp=\{\calp_n\}$ is called {\bf binary} if $\calp_1=\bfk.\id$ and $\calp_n,
n\geq 3$ are induced from $\calp_2$ by composition. Then in particular, for the free operad, we have
\begin{equation}
\calp_3=(\calp_2 \circ_1 \calp_2) \oplus (\calp_2\circ_2 \calp_2),
\mlabel{eq:bq}
\end{equation}
which can be identified with $\calp_2^{\ot 2}\oplus \calp_2^{\ot 2}$. A binary ns operad $\calp$ is called {\bf quadratic} if
all relations among the binary operations in $\calp_2$ are derived
from $\calp_3$.

Thus a binary, quadratic, ns operad
is determined by a pair $(\dfgen,\dfrel)$ where $\dfgen=\calp_2$, called
the {\bf space of generators},
and $\dfrel$ is a subspace of $\dfgen^{\otimes 2} \oplus \dfgen^{\otimes 2}$,
called the {\bf space of relations.} So we can denote $\calp=\calp(\dfgen)/(\dfrel)$.

Note that a typical element of $\dfgen^{\ot 2}$ is of the form $\sum\limits_{i=1}^k\dfoa_i\otimes \dfob_i$ with $\dfoa_i,\dfob_i\in V, 1\leq i\leq k$. Thus a typical element of
$\dfgen^{\otimes 2} \oplus \dfgen^{\otimes 2}$ is of the form
$$\left (\sum_{i=1}^k\dfoa_i\otimes \dfob_i, \sum_{j=1}^m\dfoc_j\otimes \dfod_j \right), \quad
\dfoa_i,\dfob_i,\dfoc_j,\dfod_j\in \dfgen,
1\leq i\leq k, 1\leq j\leq m, k, m\geq 1.$$

For a given binary quadratic ns operad $\calp=\calp(\dfgen)/(\dfrel)$, a $\bfk$-vector space $A$ is called a {\bf $\calp$-algebra}
\index{$\calp$-algebra}
if $A$ has binary operations (indexed by) $\dfgen$ and if,
for
$$\big (\sum_{i=1}^k\dfoa_i\otimes \dfob_i, \sum_{j=1}^m\dfoc_j\otimes \dfod_j \big)
\in \dfrel \subseteq
\dfgen^{\otimes 2} \oplus \dfgen^{\otimes 2}$$ with
$\dfoa_i,\dfob_i,\dfoc_j,\dfod_j\in \dfgen$, $1\leq i\leq k$, $1\leq j\leq m$,
we have
\begin{equation}
 \sum_{i=1}^k(x\dfoa_i y) \dfob_i z = \sum_{j=1}^m x \dfoc_j (y \dfod_j z), \quad  \forall\ x,y,z\in A.
\mlabel{eq:rel}
\end{equation}
For example, from Eq.~(\mref{eq:ns}) the NS algebras are precisely the $\calp$-algebras where $\calp=\calp(V)/(R)$ with $R$ being the subspace of $V^{\ot 2}\oplus V^{\ot 2}$ spanned by the four elements
\begin{eqnarray*}
&(\prec\ot\prec,\prec \ot\star),\quad
(\succ \ot\prec,\succ\ot\prec), & \\
&(\star \ot\succ,\succ\ot\succ),\quad
(\star\ot\bullet +\bullet \,\ot\prec, \succ\ot\bullet
+\bullet\ot\star),&
\end{eqnarray*}
where $\star=\prec+\succ+\,\bullet$.

\begin{theorem}
Let $V=\bfk\{\prec,\succ,\bullet\}$ be the vector space with basis $\{\prec,\succ,\bullet\}$ and let $\calp=\calp(\dfgen)/(\dfrel)$ be a binary quadratic ns operad. The following statements are equivalent.
\begin{enumerate}
\item
For every Nijenhuis algebra $(N,P)$, the quadruple $(N,\prec_P,\succ_P,\bullet_P)$ is a $\calp$-algebra.
\mlabel{it:nap}
\item
The relation space $\dfrel$ of $\calp$ is contained in the subspace of $V^{\ot 2}\oplus V^{\ot 2}$ spanned by
\begin{eqnarray}
&&(\prec \ot \prec,\prec\ot \star),\notag\\
&&(\succ \ot \prec,\succ\ot \prec),\notag\\
&&(\succ\ot \star, \succ\ot \succ),
\mlabel{eq:wdn}\\
&&(\prec\ot\bullet, \bullet\,\ot\succ),\notag\\
&&(\succ\ot\bullet+\bullet\ot\prec+\bullet\ot\bullet,
\succ\ot\bullet+\bullet\ot\prec+\bullet\ot\bullet), \notag
\end{eqnarray}
where $\star=\prec+\succ+\,\bullet\,$.
More precisely, any $\calp$-algebra $A$ satisfies the relations
\begin{eqnarray}
(x\prec y)\prec z&=&x \prec(y \prec z)+x \prec(y\prec
z)+x\prec(y\bullet z), \notag\\
(x\succ y)\prec z&=&x\succ(y\prec z), \notag\\
(x\prec y)\succ z+(x\succ y)\succ z+(x\bullet
 y)\succ z&=&x\succ(y\succ z), \quad \forall x,y,z\in A \mlabel{eq:wdna}\\
(x\prec y)\bullet z&=&x\bullet (y\succ z), \notag\\
(x\succ y)\bullet z+(x\bullet y)\prec z+(x\bullet y)\bullet
 z&=&x\succ(y\bullet z)+x\bullet (y\prec z)+x\bullet (y\bullet z). \notag
\end{eqnarray}
\mlabel{it:pna}
\end{enumerate}
\mlabel{thm:wdn}
\end{theorem}

Note that the relations of the NS algebra in Eq.~(\mref{eq:ns}) is contained in the space spanned by the relations in Eq.~(\mref{eq:wdn}). We call $\calp$ defined by the relations in Eq.~(\mref{eq:wdn}) the {\bf N-dendriform operad} and call a quadruple $(A,\prec,\succ,\bullet)$ satisfying Eq.~(\mref{eq:wdna}) an {\bf N-dendriform algebra}. Let $\ND$ denote the category of N-dendriform algebras. Then we have the following immediate corollary of Theorem~\mref{thm:wdn}.
\begin{coro}
\begin{enumerate}
\item
There is a natural functor
\begin{equation}
\calf: \NA \to \ND, \quad (N,P)\mapsto (N,\prec_P,\succ_P,\bullet_P).
\mlabel{eq:nasdn}
\end{equation}
\mlabel{it:nawdn}
\item
There is a natural (inclusion) functor
\begin{equation}
\calg: \ND\to \NS, \quad (M,\prec,\succ,\bullet) \mapsto (M,\prec,\succ,\bullet).
\mlabel{eq:sdndn}
\end{equation}
\mlabel{it:dnwdn}
\item
The functors $\calf$ and $\calg$ give a refinement of the functor $\cale:\NA\to \NS$ in Eq.~(\mref{eq:ns}) in the sense that the following diagram commutes
\begin{equation}
\xymatrix{ \NA \ar[rr]^{\calf} \ar[rrd]^{\cale} && \ND \ar[d]^{\calg} \\
&& \NS
}
\end{equation}
\mlabel{it:comm}
\end{enumerate}
\end{coro}

\subsection{The proof of Theorem~\mref{thm:wdn}}

With $V=\bfk \{\prec, \succ, \bullet\}$, we have
$$V^{\ot 2} \oplus V^{\ot 2} = \bigoplus\limits_{\dfop_1,\dfop_2,\dfop_3,\dfop_4\in \{\prec,\succ,\bullet\}} \bfk (\dfop_1\ot\dfop_2, \dfop_3\ot\dfop_4).$$
Thus any element $r$ of $\dfgen^{\otimes 2} \oplus \dfgen^{\otimes 2}$ is of the form
\begin{eqnarray*}
r&:=&a_1(\prec\ot \prec,0)+a_2(\prec\ot\succ,0)+a_3(\prec\ot \bullet,0)\\
&&+b_1(\succ\ot \prec,0)+b_2(\succ\ot \succ,0)+b_3(\succ\ot\bullet,0)\\
&&+c_1(\bullet\,\ot \prec,0)+c_2(\bullet\,\ot\succ,0)+c_3(\bullet\ot\bullet,0)\\
&&+d_1(0,\prec\ot \prec)+d_2(0,\prec\ot \succ)+d_3(0,\prec\ot \bullet)\\
&&+e_1(0,\succ\ot \prec)+e_3(0,\succ\ot \succ)+e_3(0,\succ\ot\bullet)\\
&&+f_1(0,\bullet\ot\prec)+f_2(0,\bullet\ot\succ) +f_3(0,\bullet\ot\bullet)
\end{eqnarray*}
where the coefficients are in $\bfk$.
\smallskip

\noindent
(\mref{it:nap} $\Rightarrow$ \mref{it:pna})
Let $\calp=\calp(V)/(R)$ be an operad satisfying the condition in Item~\mref{it:nap}. Let $r$ be in $R$ expressed in the above form. Then for any Nijenhuis algebra $(N,P)$, the quadruple $(N,\prec_P,\succ_P,\bullet_P)$ is a $\calp$-algebra. Thus \begin{eqnarray*}
&&a_1(x\prec_P y)\prec_P z+a_2(x\prec_P y)\succ_P z +a_3(x\prec_P y)\bullet_P z \\
&& +b_1(x\succ_P y)\prec_P z+b_2(x\succ_P y)\succ_P z+b_3(x\succ_P y)\bullet_P
z\\
&&+c_1(x\bullet_P y)\prec_P z+c_2(x\bullet_P y)\succ_P z+c_3(x\bullet_P y)\bullet_P z\\
&&+d_1 x\succ_P(y\succ_P z)+d_2 x\succ_P(y\prec_P z)+d_3 x\succ_P(y\bullet_P z)\\
&&+e_1 x\prec_P(y\succ_P z)+e_2 x\prec_P(y\prec_P z)+e_3x\prec_P(y\bullet_P z) \\
&&+f_1x \bullet_P(y\succ_P z)+f_2x \bullet_P (y\prec_P z)+f_3x \bullet_P (y\bullet_P z)
=0, \forall x,y,z\in N.
\end{eqnarray*}
By the definitions of $\prec_P,\succ_P,\bullet_P$ in Eq.(\mref{eq:eqs}), we have

\begin{eqnarray*}
&&a_1xP(y)P(z)+a_2P(xP(y))z-a_3P(xP(y)z)+b_1P(x)yP(z) \\
&&+b_2P(P(x)y)z-b_3P(P(x)yz)-c_1P(xy)P(z)-c_2P(P(xy))z \\ &&+c_3P(P(xy)z)+d_1P(x)P(y)z+d_2P(x)yP(z) \\
&&-d_3P(x)P(yz)+e_1xP(P(y)z)+e_2xP(yP(z))-e_3xP(P(yz)) \\ &&-f_1P(xP(y)z)-f_2P(xyP(z))+f_3P(xP(yz))
=0.
\end{eqnarray*}
Since $P$ is a Nijenhuis operator, we further have

\begin{eqnarray*}
&&a_1xP(yP(z))+a_1xP(P(y)z)-a_1xP^{2}(yz)+a_2P(xP(y))z -a_3P(xP(y)z) \\
&&+b_1P(x)yP(z)+b_2P(P(x)y)z-b_3P(P(x)yz)\\
&&-c_1P(xyP(z))-c_1P(P(xy)z)+c_1P^{2}(xyz) -c_2P(P(xy))z+c_3P(P(xy)z)\\
&& +d_1P(x)P(y)z+d_2P(x)yP(z)-d_3P(xP(yz))-d_3P(P(x)yz))
+d_3P^{2}(xyz)\\
&&+e_1xP(P(y)z)+e_2xP(yP(z))-e_3xP(P(yz))\\
&&-f_1P(xP(y)z)-f_2P(xyP(z))+f_3P(xP(yz))=0.
\end{eqnarray*}
Collecting similar terms, we obtain

\begin{eqnarray*}
&&(a_1+e_2) xP(yP( z) ) +(a_1+e_1) xP( P( y) z) -(a_1+e_3) xP(P( yz)  +(a_2+d_1) P( xP( y) ) z\\
&&-(a_3+f_1) P( xP( y) z) +(b_1+d_2) P( x) yP( z)  +(b_2+d_1) P( P( x) y) z-(b_3+d_3) P( P( x) yz) \\
&&-(c_1+f_2) P( xyP( z) ) +(c_3-c_1) P( P( xy) z) +(c_1+d_3) P^2( xyz)  \\
&&-(c_2+d_1) P( P( xy) ) z+(f_3-d_3) P( xP( yz) ) =0.
\end{eqnarray*}

Now we take the special case when $(N,P)$ is the free Nijenhuis algebra $(\FN(T(M)),P_{T(M)})$ defined in Corollary~\mref{co:vecfree} for our choice of $M=\bfk\{x,y,z\}$ and $P_{T(M)}(u)=\lc u\rc$. Then the above equation is just
\begin{eqnarray*}
&&(a_1+e_2) x\lc y\lc z\rc \rc +(a_1+e_1) x\lc \lc y\rc z\rc -(a_1+e_3) x\lc\lc yz\rc  +(a_2+d_1) \lc x\lc y\rc \rc z\\
&&-(a_3+f_1) \lc x\lc y\rc z\rc +(b_1+d_2) \lc x\rc y\lc z\rc  +(b_2+d_1) \lc \lc x\rc y\rc z-(b_3+d_3) \lc \lc x\rc yz\rc \\
&&-(c_1+f_2) \lc xy\lc z\rc \rc +(c_3-c_1) \lc \lc xy\rc z\rc +(c_1+d_3) \lc\lc xyz\rc\rc  \\
&&-(c_2+d_1) \lc \lc xy\rc \rc z+(f_3-d_3) \lc x\lc yz\rc \rc =0.
\end{eqnarray*}
Note that the set of elements
$$x\lc y\lc z\rc \rc, x\lc \lc y\rc z\rc, x\lc\lc yz\rc, \lc x\lc y\rc \rc z, \lc x\lc y\rc z\rc, \lc x\rc y\lc z\rc,$$
$$
\lc \lc x\rc y\rc z, \lc \lc x\rc yz\rc, \lc xy\lc z\rc \rc, \lc \lc xy\rc z\rc, \lc\lc xyz\rc\rc, \lc \lc xy\rc \rc z, \lc x\lc yz\rc \rc $$
is a subset of the basis $\frakX_\infty$ of the free Nijenhuis algebra $\FN(T(M))$ and hence is linearly independent. Thus the coefficients must be zero, that is,

\begin{eqnarray*}
&&a_1=-e_1=-e_2=-e_3, \\
&&a_2=b_2=c_2=-d_1, \\
&&a_3=-f_1,b_1=-d_2, \\
&&b_3=c_1=c_3=-f_2=-f_3=-d_3.
\end{eqnarray*}

Substituting these equations into the general relation $r$, we find that the any relation $r$ that can be satisfied by $\prec_P,\succ_P,\bullet_P$ for all Nijenhuis algebras $(N,P)$ is of the form

\begin{eqnarray*}
r&=&a_1\Big((x\prec y)\prec z-x \prec(y \prec z)-x \prec(y\prec
z)-x\prec(y\bullet z)\Big)\\
&&+b_1\Big((x\succ y)\prec z-x\succ(y\prec z)\Big)\\
&&+d_1\Big(x\succ(y\succ z)-(x\prec y)\succ z-(x\succ y)\succ z-(x\bullet
 y)\succ z\Big)\\
&&+ a_3\Big((x\prec y)\bullet z-x\bullet (y\succ z)\Big)\\
&&+b_3\Big((x\succ y)\bullet z+(x\bullet y)\prec z+(x\bullet y)\bullet
 z-x\succ(y\bullet z)-x\bullet (y\prec z)-x\bullet (y\bullet z)\Big),
\end{eqnarray*}
where $a_1,b_1,d_1,a_3,b_3\in \bfk$ can be arbitrary. Thus $r$ is in the subspace prescribed in Item~\mref{it:pna}, as needed.
\smallskip

\noindent
(\mref{it:pna} $\Rightarrow$ \mref{it:nap})
We check directly that all the relations in Eq.~(\mref{eq:wdna}) are satisfied by $(N,\prec_P,\succ_P,\bullet_P)$ for every
Nijenhuis algebra $(N,P)$. First of all
\begin{eqnarray*}
(x\prec_P y) \prec_P z&=&xP(y)P(z) \\
&=&xP(yP(z))+xP(P(y)z)-xP^2(yz)\\
&=&x\prec_P (y\prec_P z)+x\prec_P (y\succ_P z)
+x\prec_P (y\bullet_P z),
\end{eqnarray*}
proving the first equation in Eq.~(\mref{eq:wdna}). The proofs of the second and third equations are similar. For the fourth equation, we have
$$ (x\prec_P y)\bullet_P z=-P((xP(y))z)=-P(x(P(y)z))=x\bullet_P (y\succ_P z).$$
Finally for the last equation, we verify
\begin{eqnarray*}
&&(x\succ_P y)\bullet_P z+(x\bullet_P y)\prec_P z +(x\bullet_P y)\bullet_P z \\
&=& -P((P(x)y)z)-P(xy)P(z)+P(P(xy)z)\\
&=& -P(P(x)yz)-P(xyP(z))-P(P(xy)z)+P^2(xyz)+P(P(xy)z)\\
&=& -P(P(x)yz)-P(xyP(z))+P^2(xyz),
\end{eqnarray*}
and
\begin{eqnarray*}
&& x\succ_P (y\bullet_P z)+x\bullet_P(y\prec_P z) +x\bullet_P(y\bullet_P z)\\
&=& -P(x)P(yz) -P(x(yP(z))) +P(xP(yz))\\
&=& -P(xP(yz))-P(P(x)yz)+P^2(xyz)-P(xyP(z))+P(xP(yz))\\
&=& -P(P(x)yz)+P^2(xyz)-P(xyP(z)).
\end{eqnarray*}
So the two sides of the last equation agree.

Thus if the relation space $R$ of an operad $\calp=\calp(V)/(R)$ is contained in the subspace spanned by the vectors in Eq.~(\mref{eq:wdn}), then the corresponding relations are linear combinations of the equations in Eq.~(\mref{eq:wdna}) and hence are satisfied by $(N,\prec_P,\succ_P,\bullet_P)$ for each Nijenhuis algebra $(N,P)$. Therefore $(N,\prec_P,\succ_P,\bullet_P)$ is a $\calp$-algebra. This completes the proof of Theorem~\mref{thm:wdn}.


\end{document}